\documentclass[12pt]{amsart}
 
\usepackage{amssymb, amscd, txfonts}
\usepackage{graphicx}
\usepackage[all]{xy} 
\usepackage{array, booktabs}
\usepackage{tabularx} 

\numberwithin{equation}{section}

\newtheorem{theorem}{Theorem}[section]
\newtheorem{proposition}[theorem]{Proposition}
\newtheorem{lemma}[theorem]{Lemma}
\newtheorem{corollary}[theorem]{Corollary}

\theoremstyle{definition}

\newtheorem{condition}[theorem]{Condition}

\theoremstyle{remark}
\newtheorem{remark}[theorem]{Remark}
\newtheorem{claim}[theorem]{Claim}

\newcommand{\Z}{\mathbb{Z}}
\newcommand{\Q}{\mathbb{Q}}

\newcommand{\C}{\mathbb{C}}
\newcommand{\proj}{{\mathbb P}}

\newcommand{\Fgn}{\mathcal{F}_{g,n}}
\newcommand{\Fg}{\mathcal{F}_{g}}
\newcommand{\Fgncpt}{\bar{\mathcal{F}}_{g,n}}
\newcommand{\G}{\Gamma}
\newcommand{\Gg}{\Gamma_{g}}

\begin{document}

\title[]{Differential forms on universal $K3$ surfaces}
\author[]{Shouhei Ma}
\thanks{Supported by KAKENHI 21H00971 and 20H00112.} 
\address{Department~of~Mathematics, Tokyo~Institute~of~Technology, Tokyo 152-8551, Japan}
\email{ma@math.titech.ac.jp}

\begin{abstract}
We give a vanishing and classification result for holomorphic differential forms 
on smooth projective models of the moduli spaces of pointed $K3$ surfaces. 
We prove that there is no nonzero holomorphic $k$-form for $0<k<10$ and for even $k>19$. 
In the remaining cases, we give an isomorphism between the space of holomorphic $k$-forms 
with that of vector-valued modular forms ($10\leq k \leq 18$) 
or scalar-valued cusp forms (odd $k\geq 19$) for the modular group. 
These results are in fact proved in the generality of lattice-polarization. 
\end{abstract} 

\maketitle

\section{Introduction}\label{sec: intro}

Let ${\Fgn}$ be the moduli space of $n$-pointed $K3$ surfaces of genus $g>2$, i.e., 
primitively polarized of degree $2g-2$. 
It is a quasi-projective variety of dimension $19+2n$ with a natural morphism ${\Fgn}\to {\Fg}$ 
to the moduli space ${\Fg}$ of $K3$ surfaces of genus $g$, which is generically a $K3^{n}$-fibration. 
In this paper we study holomorphic differential forms on a smooth projective model of ${\Fgn}$. 
They do not depend on the choice of a smooth projective model, 
and thus are fundamental birational invariants of ${\Fgn}$. 
We prove a vanishing result for about half numbers of degrees, 
and for the remaining degrees give a correspondence with modular forms on the period domain. 

Our main result is stated as follows. 

\begin{theorem}\label{thm: main}
Let ${\Fgncpt}$ be a smooth projective model of ${\Fgn}$ with $g>2$.  
Then we have a natural isomorphism 
\begin{equation}\label{eqn: main}
H^{0}({\Fgncpt}, \Omega^{k}) \simeq 
\begin{cases}
0 & \: \: 0<k\leq 9 \\ 
M_{\wedge^{k},k}({\Gg}) & \: \: 10 \leq k \leq 18 \\ 
0 & \: \: k>19, \: k\in 2{\Z} \\ 
S\!_{19+m}({\Gg}, \det)\otimes {\C}\mathcal{S}_{n,m} & \: \: k=19+2m, \: 0\leq m \leq n 
\end{cases}
\end{equation}
\end{theorem} 

Here ${\Gg}$ is the modular group for $K3$ surfaces of genus $g$, 
which is defined as the kernel of ${\rm O}^{+}(L_{g})\to {\rm O}(L_{g}^{\vee}/L_{g})$ 
where $L_{g}=2U\oplus 2E_{8}\oplus \langle 2-2g \rangle$ is the period lattice of $K3$ surfaces of genus $g$. 
In the second case, $M_{\wedge^{k},k}({\Gg})$ stands for 
the space of vector-valued modular forms of weight $(\wedge^{k},k)$ for ${\Gg}$ (\cite{Ma3}). 
In the last case, $S\!_{19+m}({\Gg}, \det)$ stands for the space of scalar-valued cusp forms of weight $19+m$ and 
determinant character for ${\Gg}$, 
and $\mathcal{S}_{n,m}$ stands for the coset $\frak{S}_{n}/(\frak{S}_{m}\times \frak{S}_{n-m})$. 
Theorem \ref{thm: main} is actually formulated and proved in the generality of lattice-polarization (Theorem \ref{thm: main lattice-pol}). 

In the case of the top degree $k=19+2n$, namely for canonical forms, the isomorphism \eqref{eqn: main} is proved in \cite{Ma1}. 
Theorem \ref{thm: main} is the extension of this result to all degrees $k<19+2n$. 
The spaces in the right hand side of \eqref{eqn: main} can also be geometrically explained as follows. 
In the case $k\leq 18$, $M_{\wedge^{k},k}({\Gg})$ is identified with 
the space of holomorphic $k$-forms on a smooth projective model of ${\Fg}$, pulled back by ${\Fgn}\to {\Fg}$. 
In the case $k=19+2m$, $S\!_{19+m}({\Gg}, \det)$ is identified with the space of canonical forms on $\bar{\mathcal{F}}_{g,m}$, 
and the tensor product $S\!_{19+m}({\Gg}, \det)\otimes {\C}\mathcal{S}_{n,m}$ 
is the direct sum of pullback of such canonical forms by various projections ${\Fgn}\to \mathcal{F}_{g,m}$. 
Therefore Theorem \ref{thm: main} can be understood as a kind of classification result which says that 
except for canonical forms, there is essentially no new differential forms on the tower $({\Fgn})_{n}$ of moduli spaces. 
In fact, this is how the proof proceeds. 

The space $S\!_{l}({\Gg}, \det)$ is nonzero for every sufficiently large $l$, 
so the space $H^{0}({\Fgncpt}, \Omega^{k})$ for odd $k\geq 19$ is typically nonzero (at least when $k$ is large). 
On the other hand, it is not clear at present whether $M_{\wedge^{k},k}({\Gg})\ne 0$ or not in the range $10\leq k \leq 18$. 
This is a subject of study in the theory of vector-valued orthogonal modular forms. 

The isomorphism \eqref{eqn: main} in the case $k=19+2m$ is an $\frak{S}_{n}$-equivariant isomorphism, 
where $\frak{S}_{n}$ acts on $H^{0}({\Fgncpt}, \Omega^{k})$ by its permutation action on ${\Fgn}$, 
while it acts on $S\!_{19+m}({\Gg}, \det)\otimes {\C}\mathcal{S}_{n,m}$ 
by its natural action on the coset $\mathcal{S}_{n,m}=\frak{S}_{n}/(\frak{S}_{m}\times \frak{S}_{n-m})$. 
Therefore, taking the $\frak{S}_{n}$-invariant part, we obtain the following simpler result 
for the unordered pointed moduli space ${\Fgn}/\frak{S}_{n}$, which is birationally a $K3^{[n]}$-fibration over ${\Fg}$. 

\begin{corollary}
Let $\overline{{\Fgn}/\frak{S}_{n}}$ be a smooth projective model of ${\Fgn}/\frak{S}_{n}$. 
Then we have a natural isomorphism 
\begin{equation*}
H^{0}(\overline{{\Fgn}/\frak{S}_{n}}, \Omega^{k}) \simeq 
\begin{cases}
0 & \: \: 0<k\leq 9 \\ 
M_{\wedge^{k},k}({\Gg}) & \: \:  10\leq k \leq 18 \\ 
0 & \: \: k>19, \: k\in 2{\Z} \\ 
S\!_{19+m}({\Gg}, \det) & \: \: k=19+2m, \: 0 \leq m \leq n 
\end{cases}
\end{equation*}
\end{corollary} 

The universal $K3$ surface $\mathcal{F}_{g,1}$ is an analogue of elliptic modular surfaces (\cite{Shi}), 
and the moduli spaces ${\Fgn}$ for general $n$ are analogues of the so-called Kuga varieties over modular curves (\cite{Sho}). 
Starting with the case of elliptic modular surfaces \cite{Shi}, 
holomorphic differential forms on the Kuga varieties have been described in terms of elliptic modular forms: 
\cite{Sho} for canonical forms, and \cite{Gor} for the case of lower degrees (somewhat implicitly). 
Theorem \ref{thm: main} can be regarded as a $K3$ version of these results. 

As a final remark, 
in view of the analogy between universal $K3$ surfaces and elliptic modular surfaces, 
invoking the classical fact that elliptic modular surfaces have maximal Picard number (\cite{Shi}) now raises the question 
if $H^{k,0}({\Fgncpt})\oplus H^{0,k}({\Fgncpt})$ is a sub ${\Q}$-Hodge structure of $H^{k}({\Fgncpt}, {\C})$. 
This is independent of the choice of a smooth projective model ${\Fgncpt}$. 

The rest of this paper is devoted to the proof of Theorem \ref{thm: main}. 
In \S \ref{ssec: hol Leray} we compute a part of the holomorphic Leray spectral sequence 
associated to a certain type of $K3^{n}$-fibration. 
This is the main step of the proof. 
In \S \ref{ssec: extension} we study differential forms on a compactification of such a fibration. 
In \S \ref{ssec: univ K3} we deduce (a generalized version of) Theorem \ref{thm: main} 
by combining the result of \S \ref{ssec: extension} 
with some results from \cite{Pom}, \cite{Ma1}, \cite{Ma2}, \cite{Ma3}. 
Sometimes we drop the subscript $X$ from the notation $\Omega_{X}^{k}$ when the variety $X$ is clear from the context.

\section{Proof}

\subsection{Holomorphic Leray spectral sequence}\label{ssec: hol Leray}

Let $\pi\colon X\to B$ be a smooth family of $K3$ surfaces over a smooth connected base $B$. 
In this subsection $X$ and $B$ may be analytic. 
We put the following assumption: 

\begin{condition}\label{condition}
In a neighborhood of every point of $B$, the period map is an embedding. 
\end{condition}

\noindent
This is equivalent to the condition that 
the differential of the period map 
\begin{equation*}
T_{b}B \to {\rm Hom}(H^{2,0}(X_{b}), H^{1,1}(X_{b})) 
\end{equation*}
is injective for every $b\in B$, where $X_{b}$ is the fiber of $\pi$ over $b$. 

For a natural number $n>0$ we denote by $X_{n}=X\times_{B}\cdots \times_{B}X$ 
the $n$-fold fiber product of $X$ over $B$, 
and let $\pi_{n}\colon X_{n}\to B$ be the projection. 
We denote by $\Omega_{\pi_{n}}$ the relative cotangent bundle of $\pi_{n}$, and 
$\Omega_{\pi_{n}}^{p}=\wedge^{p}\Omega_{\pi_{n}}$ for $p\geq 0$ as usual. 

\begin{proposition}\label{prop: hol Leray}
Let $\pi\colon X\to B$ be a $K3$ fibration satisfying Condition \ref{condition}. 
Then we have a natural isomorphism 
\begin{equation*}
(\pi_{n})_{\ast}\Omega_{X_{n}}^{k} \simeq 
\begin{cases}
\Omega_{B}^{k} & \; \; k\leq \dim B \\ 
0 & \; \; k>\dim B, \: k\not \equiv \dim B \: (2) \\ 
K_{B}\otimes (\pi_{n})_{\ast}\Omega_{\pi_{n}}^{2m} & \; \; k=\dim B+2m, \: 0\leq m \leq n 
\end{cases}
\end{equation*}
\end{proposition}

This assertion amounts to a partial degeneration of the holomorphic Leray spectral sequence. 
Recall (\cite{Voi} \S 5.2) that $\Omega_{X_{n}}^k$ has the holomorphic Leray filtration $L^{\bullet}\Omega_{X_{n}}^k$ defined by 
\begin{equation*}
L^{l}\Omega_{X_{n}}^k= \pi_{n}^{\ast}\Omega_{B}^{l}\wedge \Omega_{X_{n}}^{k-l}, 
\end{equation*}
whose graded quotients are naturally isomorphic to 
\begin{equation*}
{\rm Gr}_{L}^{l}\Omega_{X_{n}}^{k} = \pi_{n}^{\ast}\Omega_{B}^{l} \otimes \Omega_{\pi_{n}}^{k-l}. 
\end{equation*}
This filtration induces the holomorphic Leray spectral sequence 
\begin{equation*}
(E_{r}^{l,q}, d_{r}) \: \: \:  \Rightarrow \: \: \:  E_{\infty}^{l+q} = R^{l+q}(\pi_{n})_{\ast}\Omega_{X_{n}}^{k} 
\end{equation*}
which converges to the filtration 
\begin{equation*}
L^{l}R^{l+q}(\pi_{n})_{\ast}\Omega_{X_{n}}^{k} \: = \: 
{\rm Im}(R^{l+q}(\pi_{n})_{\ast}L^{l}\Omega_{X_{n}}^{k} \to R^{l+q}(\pi_{n})_{\ast}\Omega_{X_{n}}^{k}). 
\end{equation*}
By \cite{Voi} Proposition 5.9,  the $E_{1}$ page coincides with 
the collection of the Koszul complexes associated to the variation of Hodge structures for $\pi_{n}$: 
\begin{equation}\label{eqn: E1}
(E_{1}^{l,q}, d_{1}) = (\mathcal{H}^{k-l,l+q}\otimes \Omega_{B}^{l}, \bar{\nabla}). 
\end{equation}
Here $\mathcal{H}^{\ast, \ast}$ are the Hodge bundles associated to the fibration $\pi_{n}\colon X_{n}\to B$, and 
\begin{equation*}
\bar{\nabla} : \mathcal{H}^{\ast, \ast}\otimes \Omega_{B}^{\ast} \to \mathcal{H}^{\ast-1, \ast+1}\otimes \Omega_{B}^{\ast+1} 
\end{equation*}
are the differentials in the Koszul complexes (see \cite{Voi} \S 5.1.3). 
For degree reasons, the range of $(l, q)$ in the $E_{1}$ page satisfies the inequalities 
\begin{equation*}
0\leq l \leq \dim B, \quad 0\leq k-l \leq 2n, \quad 0\leq l+q \leq 2n. 
\end{equation*}
The first two can be unified: 
\begin{equation}\label{eqn: (l,q)}
\max(0, k-2n) \leq l \leq \min(\dim B, k), \quad 0\leq l+q \leq 2n. 
\end{equation}
We calculate the $E_{1}$ to $E_{2}$ pages on the edge line $l+q=0$. 

\begin{lemma}\label{lem: spectral sequence}
The following holds. 

(1) $E_{1}^{l,-l}=0$ when $l\leq \min(\dim B, k)$ with $l\not\equiv k$ mod $2$. 

(2) $E_{2}^{l,-l}=0$ when $l< \min(\dim B, k)$. 

(3) For $l_{0} = \min(\dim B, k)$ we have 
$E_{1}^{l_{0},-l_{0}}=E_{2}^{l_{0},-l_{0}}= \cdots = E_{\infty}^{l_{0},-l_{0}}$. 
\end{lemma}

\begin{proof}
By \eqref{eqn: E1}, we have $E_{1}^{l,-l}=\mathcal{H}^{k-l,0}\otimes \Omega_{B}^{l}$. 
By the K\"unneth formula, the fiber of $\mathcal{H}^{k-l,0}$ over a point $b\in B$ is identified with 
\begin{equation}\label{eqn: Hk-l,0}
H^{k-l,0}(X_{b}^{n}) = \bigoplus_{(p_{1}, \cdots, p_{n})} H^{p_{1},0}(X_{b}) \otimes \cdots \otimes H^{p_{n},0}(X_{b}), 
\end{equation} 
where $(p_{1}, \cdots, p_{n})$ ranges over all indices with 
$\sum_{i}p_{i}=k-l$ and $0\leq p_{i} \leq 2$. 

(1) When $k-l$ is odd, every index $(p_{1}, \cdots, p_{n})$ in \eqref{eqn: Hk-l,0} 
must contain a component $p_{i}=1$. 
Since $H^{1,0}(X_b)=0$, we see that $H^{k-l,0}(X_{b}^{n})=0$. 
Therefore $\mathcal{H}^{k-l,0}=0$ when $k-l$ is odd. 

(3) Let $l_{0} = \min(\dim B, k)$. 
By the range \eqref{eqn: (l,q)} of $(l, q)$, 
we see that for every $r\geq 1$ 
the source of $d_{r}$ that hits to $E_{r}^{l_{0}, -l_{0}}$ is zero, 
and the target of $d_{r}$ that starts from $E_{r}^{l_{0}, -l_{0}}$ is also zero. 
This proves our assertion. 

(2) Let $l < \min(\dim B, k)$. 
In view of (1), we may assume that $l=k-2m$ for some $m>0$. 
By \eqref{eqn: (l,q)}, the source of $d_{1}$ that hits to $E_{1}^{l,-l}$ is zero. 
We shall show that 
$d_{1}\colon E_{1}^{l,-l}\to E_{1}^{l+1,-l}$ 
is injective. 
By \eqref{eqn: E1}, this morphism is identified with 
\begin{equation}\label{eqn: d1 Koszul}
\bar{\nabla} : \mathcal{H}^{2m,0}\otimes \Omega_{B}^{l} \to  \mathcal{H}^{2m-1,1}\otimes \Omega_{B}^{l+1}. 
\end{equation}
By the K\"unneth formula as in \eqref{eqn: Hk-l,0}, 
the fibers of the Hodge bundles $\mathcal{H}^{2m,0}$, $\mathcal{H}^{2m-1,1}$ over $b\in B$ 
are respectively identified with 
\begin{equation}\label{eqn: H2m,0}
H^{2m,0}(X_{b}^{n}) = \bigoplus_{|\sigma|=m} H^{2,0}(X_{b})^{\otimes \sigma}, 
\end{equation}
\begin{eqnarray}\label{eqn: H2m-1,1}
H^{2m-1,1}(X_{b}^{n}) 
& = & \bigoplus_{|\sigma'|=m-1} \bigoplus_{i\not\in \sigma'} H^{2,0}(X_{b})^{\otimes \sigma'}\otimes H^{1,1}(X_{b}) \\ 
& = & \bigoplus_{|\sigma|=m} \bigoplus_{i\in \sigma} H^{2,0}(X_{b})^{\otimes \sigma-\{ i \}}\otimes H^{1,1}(X_{b}). \nonumber 
\end{eqnarray}
In \eqref{eqn: H2m,0}, $\sigma$ ranges over 
all subsets of $\{ 1, \cdots, n\}$ consisting of $m$ elements, 
and $H^{2,0}(X_{b})^{\otimes \sigma}$ stands for 
the tensor product of $H^{2,0}(X_{b})$ for the $j$-th factors $X_{b}$ of $X_{b}^{n}$ over all $j\in \sigma$.  
The notations $\sigma', \sigma$ in \eqref{eqn: H2m-1,1} are similar, 
and $H^{1,1}(X_{b})$ in \eqref{eqn: H2m-1,1} is the $H^{1,1}$ of the $i$-th factor $X_{b}$ of $X_{b}^{n}$. 

Let us write $V=H^{2,0}(X_{b})$ and $W=(T_{b}B)^{\vee}$ for simplicity. 
The homomorphism \eqref{eqn: d1 Koszul} over $b\in B$ is written as 
\begin{equation}\label{eqn: d1 Hodge}
\bigoplus_{|\sigma|=m} \left( V^{\otimes \sigma}\otimes \wedge^{l}W \to 
\bigoplus_{i\in \sigma} V^{\otimes \sigma - \{ i\}}\otimes H^{1,1}(X_{b})\otimes \wedge^{l+1}W \right). 
\end{equation}
By \cite{Voi} Lemma 5.8, the $(\sigma, i)$-component 
\begin{equation}\label{eqn: (sigma, i)}
V^{\otimes \sigma}\otimes \wedge^{l}W \to 
V^{\otimes \sigma - \{ i\}}\otimes H^{1,1}(X_{b})\otimes \wedge^{l+1}W 
\end{equation}
factorizes as 
\begin{eqnarray*}
V^{\otimes \sigma}\otimes \wedge^{l}W & \to & 
V^{\otimes \sigma - \{ i\}}\otimes H^{1,1}(X_{b}) \otimes W \otimes \wedge^{l}W \\ 
& \to & V^{\otimes \sigma - \{ i\}}\otimes H^{1,1}(X_{b})\otimes \wedge^{l+1}W, 
\end{eqnarray*} 
where the first map is induced by the adjunction 
$V\to H^{1,1}(X_b)\otimes W$ of the differential of the period map for the $i$-th factor $X_{b}$, 
and the second map is induced by the wedge product 
$W \otimes \wedge^{l}W \to \wedge^{l+1}W$. 
By linear algebra, this composition can also be decomposed as  
\begin{eqnarray}\label{eqn: decompose}
V^{\otimes \sigma}\otimes \wedge^{l}W & \to & 
V^{\otimes \sigma - \{ i\}}\otimes V \otimes W^{\vee} \otimes \wedge^{l+1}W \\ 
& \to & V^{\otimes \sigma - \{ i\}}\otimes H^{1,1}(X_{b})\otimes \wedge^{l+1}W, \nonumber 
\end{eqnarray} 
where the first map is induced by the adjunction $\wedge^{l}W \to W^{\vee} \otimes \wedge^{l+1}W$ 
of the wedge product, 
and the second map is induced by the adjunction $V\otimes W^{\vee}\to H^{1,1}(X_{b})$ of the differential of the period map. 
By our initial assumption \ref{condition}, the second map of \eqref{eqn: decompose} is injective. 
Moreover, since $l+1\leq \dim W$ by our assumption, 
the wedge product $\wedge^{l}W\times W \to \wedge^{l+1}W$ is nondegenerate, 
so its adjunction $\wedge^{l}W \to W^{\vee} \otimes \wedge^{l+1}W$ is injective. 
Thus the first map of \eqref{eqn: decompose} is also injective. 
It follows that \eqref{eqn: (sigma, i)} is injective. 
Since the map \eqref{eqn: d1 Hodge} is the direct sum of its $(\sigma, i)$-components, it is injective. 
This finishes the proof of Lemma \ref{lem: spectral sequence}. 
\end{proof}

We can now complete the proof of Proposition \ref{prop: hol Leray}. 

\begin{proof}[(Proof of Proposition \ref{prop: hol Leray})]
By Lemma \ref{lem: spectral sequence} (2), 
we have $E_{\infty}^{l,-l}=0$ when $l<l_{0}=\min(\dim B, k)$. 
Together with Lemma \ref{lem: spectral sequence} (3), we obtain  
\begin{equation*}
(\pi_{n})_{\ast}\Omega_{X_{n}}^{k} = E_{\infty}^{0} = E_{\infty}^{l_{0}, -l_{0}} = E_{1}^{l_{0}, -l_{0}}. 
\end{equation*}
When $k\leq \dim B$, we have $l_{0}=k$, and 
$E_{1}^{l_{0}, -l_{0}}=\Omega_{B}^{k}$ by \eqref{eqn: E1}. 
When $k> \dim B$, we have $l_{0}= \dim B$, and 
$E_{1}^{l_{0}, -l_{0}}=\mathcal{H}^{k-\dim B, 0}\otimes K_{B}$ by \eqref{eqn: E1}. 
When $k-\dim B$ is odd, this vanishes by Lemma \ref{lem: spectral sequence} (1). 
\end{proof}

In the case $k=\dim B + 2m$, the vector bundle 
$\mathcal{H}^{2m,0}\otimes K_{B}=(\pi_{n})_{\ast}\Omega_{\pi_{n}}^{2m}\otimes K_{B}$ 
can be written more specifically as follows. 
For a subset $\sigma$ of $\{ 1, \cdots, n \}$ with cardinality $| \sigma |=m$, 
we denote by $X_{\sigma}\simeq X_{m}$ the fiber product of 
the $i$-th factors $X\to B$ of $X_{n}\to B$ over all $i\in \sigma$. 
We denote by 
\begin{equation*}
X_{n} \stackrel{\pi_{\sigma}}{\to} X_{\sigma} \stackrel{\pi^{\sigma}}{\to} B 
\end{equation*}
the natural projections. 
The K\"unneth formula \eqref{eqn: H2m,0} says that 
\begin{equation*}
(\pi_{n})_{\ast}\Omega_{\pi_{n}}^{2m} \simeq 
\bigoplus_{|\sigma|=m} \pi^{\sigma}_{\ast}K_{\pi^{\sigma}}. 
\end{equation*}
Combining this with the isomorphism 
\begin{equation}\label{eqn: KXsigma}
\pi^{\sigma}_{\ast}K_{X_{\sigma}}\simeq K_{B}\otimes \pi^{\sigma}_{\ast}K_{\pi^{\sigma}} 
\end{equation}
for each $X_{\sigma}$, we can rewrite the isomorphism in the last case of Proposition \ref{prop: hol Leray} as  
\begin{equation}\label{eqn: push OmegaX k>dimB}
(\pi_{n})_{\ast}\Omega_{X_{n}}^{\dim B+2m} \simeq 
\bigoplus_{|\sigma|=m} \pi^{\sigma}_{\ast}K_{X_{\sigma}}. 
\end{equation}

\subsection{Extension over compactification}\label{ssec: extension}

Let $\pi\colon X\to B$ be a $K3$ fibration as in \S \ref{ssec: hol Leray}. 
We now assume that $X, B$ are quasi-projective and $\pi$ is a morphism of algebraic varieties. 
We take smooth projective compactifications of $X_{n}, X_{\sigma}, B$ and denote them by 
$\bar{X}_{n}, \bar{X}_{\sigma}, \bar{B}$ respectively. 

\begin{proposition}\label{prop: extension}
We have 
\begin{equation*}
H^{0}(\bar{X}_{n}, \Omega^{k}) \simeq 
\begin{cases}
H^{0}(\bar{B}, \Omega^{k}) & \: \: k\leq \dim B \\ 
0 & \: \: k>\dim B, \: k\not\equiv \dim B \: (2) \\ 
\oplus_{\sigma}H^{0}(\bar{X}_{\sigma}, K_{\bar{X}_{\sigma}}) & \: \: k=\dim B+2m, \: 0\leq m \leq n  
\end{cases}
\end{equation*}
In the last case, $\sigma$ ranges over all subsets of $\{ 1, \cdots, n \}$ with $|\sigma|=m$. 
The isomorphism in the first case is given by the pullback by $\pi_{n}\colon X_{n}\to B$, 
and the isomorphism in the last case is given by the direct sum of the pullbacks by 
$\pi_{\sigma}\colon X_{n}\to X_{\sigma}$ for all $\sigma$. 
\end{proposition}

\begin{proof}
The assertion in the case $k>\dim B$ with $k\not\equiv \dim B$ mod $2$ follows directly from 
the second case of Proposition \ref{prop: hol Leray}. 
Next we consider the case $k\leq \dim B$. 
We may assume that $\pi_{n}\colon X_{n}\to B$ extends to a surjective morphism $\bar{X}_{n}\to \bar{B}$. 
Let $\omega$ be a holomorphic $k$-form on $\bar{X}_{n}$. 
By the first case of Proposition \ref{prop: hol Leray}, we have 
$\omega|_{X_{n}}=\pi_{n}^{\ast}\omega_{B}$ for a holomorphic $k$-form $\omega_{B}$ on $B$. 
Since $\omega$ is holomorphic over $\bar{X}_{n}$, 
$\omega_{B}$ is holomorphic over $\bar{B}$ as well by a standard property of holomorphic differential forms. 
(Otherwise $\omega$ must have pole at the divisors of $\bar{X}_{n}$ dominating 
the divisors of $\bar{B}$ where $\omega_{B}$ has pole.) 
Therefore the pullback 
$H^{0}(\bar{B}, \Omega^{k})\to H^{0}(\bar{X}_{n}, \Omega^{k})$ 
is surjective. 

Finally, we consider the case $k=\dim B+2m$, $0\leq m \leq n$. 
Let $\omega$ be a holomorphic $k$-form on $\bar{X}_{n}$. 
By \eqref{eqn: push OmegaX k>dimB}, we can uniquely write 
$\omega|_{X_{n}}=\sum_{\sigma}\pi_{\sigma}^{\ast}\omega_{\sigma}$ 
for some canonical forms $\omega_{\sigma}$ on $X_{\sigma}$. 

\begin{claim}
For each $\sigma$, $\omega_{\sigma}$ is holomorphic over $\bar{X}_{\sigma}$. 
\end{claim}

\begin{proof}
We identify $X_{n}$ with the fiber product $X_{\sigma}\times_{B}X_{\tau}$ 
where $\tau=\{ 1, \cdots, n\} - \sigma$ is the coset of $\sigma$. 
We may assume that this fiber product diagram extends to a commutative diagram of surjective morphisms 
\begin{equation*}
\begin{CD}
     \bar{X}_{n} @>{\pi_{\tau}}>> \bar{X}_{\tau} \\
  @V{\pi_{\sigma}}VV    @VV{\pi^{\tau}}V \\
     \bar{X}_{\sigma}   @>{\pi^{\sigma}}>>  \bar{B} 
\end{CD}
\end{equation*}
between smooth projective models. 
We take an irreducible subvariety $\tilde{B}\subset \bar{X}_{\tau}$ such that 
$\tilde{B}\to \bar{B}$ is surjective and generically finite. 
Then $\pi_{\tau}^{-1}(\tilde{B})\subset \bar{X}_{n}$ has a unique irreducible component dominating $\tilde{B}$. 
We take its desingularization and denote it by $Y$. 
By construction $\pi_{\sigma}|_{Y} \colon Y\to \bar{X}_{\sigma}$ is dominant (and so surjective) and generically finite. 
On the other hand, for any $\sigma'\ne \sigma$ with $|\sigma'|=m$, 
the projection $\pi_{\sigma'}|_{Y} \colon Y\dashrightarrow X_{\sigma'}$ is not dominant. 
Indeed, such $\sigma'$ contains at least one component $i\in \tau$, 
so if $Y\dashrightarrow X_{\sigma'}$ was dominant, 
then the $i$-th projection $Y\dashrightarrow X$ would be also dominant, 
which is absurd because it factorizes as $Y\to \tilde{B}\subset \bar{X}_{\tau}\dashrightarrow X$. 

We pullback the differential form 
$\omega=\pi_{\sigma}^{\ast}\omega_{\sigma}+\sum_{\sigma'\ne \sigma}\pi_{\sigma'}^{\ast}\omega_{\sigma'}$ 
to $Y$ and denote it by $\omega|_{Y}$. 
Since $\omega$ is holomorphic over $\bar{X}_{n}$, $\omega|_{Y}$ is holomorphic over $Y$. 
Since $\pi_{\sigma'}^{\ast}\omega_{\sigma'}|_{Y}$ is the pullback of the canonical form $\omega_{\sigma'}$ 
on $X_{\sigma'}$ by the non-dominant map $Y \dashrightarrow X_{\sigma'}$, it vanishes identically. 
Hence $\pi_{\sigma}^{\ast}\omega_{\sigma}|_{Y}=\omega|_{Y}$ is holomorphic over $Y$. 
Since $\pi_{\sigma}|_{Y}\colon Y \to \bar{X}_{\sigma}$ is surjective, 
this implies that $\omega_{\sigma}$ is holomorphic over $\bar{X}_{\sigma}$ as before. 
\end{proof}

The above argument will be clear if we consider over the generic point $\eta$ of $B$: 
we restrict $\omega$ to the fiber of $(X_{\eta})^{n}\to (X_{\eta})^{\tau}$ over 
the geometric point $\tilde{B}$ of $(X_{\eta})^{\tau}$ over $\eta$. 

By this claim, the pullback 
\begin{equation*}
(\pi_{\sigma}^{\ast})_{\sigma} : 
\bigoplus_{|\sigma|=m} H^{0}(\bar{X}_{\sigma}, K_{\bar{X}_{\sigma}}) \to H^{0}(\bar{X}_{n}, \Omega^{\dim B + 2m}) 
\end{equation*}
is surjective. 
It is also injective as implied by \eqref{eqn: push OmegaX k>dimB}. 
This proves Proposition \ref{prop: extension}. 
\end{proof}

\subsection{Universal $K3$ surface}\label{ssec: univ K3}

Now we prove Theorem \ref{thm: main}, in the generality of lattice-polarization. 
Let $L$ be an even lattice of signature $(2, d)$ 
which can be embedded as a primitive sublattice of the $K3$ lattice $3U\oplus 2E_{8}$. 
We denote by  
\begin{equation*}
\mathcal{D} = \{ \: {\C} \omega \in {\proj}L_{{\C}} \: | \: (\omega, \omega)=0, (\omega, \bar{\omega})>0 \: \}^{+} 
\end{equation*}
the Hermitian symmetric domain associated to $L$, where $+$ means a connected component. 

Let $\pi\colon X\to B$ be a smooth projective family of $K3$ surfaces over a smooth quasi-projective connected base $B$. 
We say (\cite{Ma2}) that the family $\pi\colon X\to B$ is \textit{lattice-polarized with period lattice} $L$ 
if there exists a sub local system $\Lambda$ of $R^{2}\pi_{\ast}{\Z}$ such that 
each fiber $\Lambda_{b}$ is a hyperbolic sublattice of the N\'eron-Severi lattice $NS(X_{b})$ and 
the fibers of the orthogonal complement $\Lambda^{\perp}$ are isometric to $L$. 
Then we have a period map 
\begin{equation*}
\mathcal{P} : B \to {\G}\backslash \mathcal{D} 
\end{equation*}
for some finite-index subgroup ${\G}$ of ${\rm O}^{+}(L)$. 
By Borel's extension theorem, $\mathcal{P}$ is a morphism of algebraic varieties. 

Let us put the assumption 
\begin{equation}\label{eqn: period map birational}
\mathcal{P} \; \textrm{is birational and} \: -{\rm id}\not\in {\G}. 
\end{equation}
For such a family $\pi\colon X\to B$, if we shrink $B$ as necessary, 
then $\mathcal{P}$ is an open immersion and Condition \ref{condition} is satisfied. 
For example, the universal $K3$ surface $\mathcal{F}_{g,1}\to {\Fg}$ for $g>2$ 
restricted over a Zariski open set of ${\Fg}$ satisfies this assumption with 
$L=L_{g}$ and ${\G}={\Gg}$ (see \S \ref{sec: intro} for these notations). 

As in \S \ref{sec: intro}, we denote by 
$M_{\wedge^{k},k}({\G})$ the space of vector-valued modular forms of weight $(\wedge^{k},k)$ for ${\G}$,  
$S\!_{l}({\G}, \det)$ the space of scalar-valued cusp forms of weight $l$ and character $\det$ for ${\G}$, 
and $\mathcal{S}_{n,m}=\frak{S}_{n}/(\frak{S}_{m}\times \frak{S}_{n-m})$. 

\begin{theorem}\label{thm: main lattice-pol}
Let $\pi\colon X\to B$ be a lattice-polarized $K3$ family with period lattice $L$ of signature $(2, d)$ with $d\geq 3$ 
and monodromy group ${\G}$ satisfying \eqref{eqn: period map birational}. 
Then we have an $\frak{S}_{n}$-equivariant isomorphism 
\begin{equation*}
H^{0}(\bar{X}_{n}, \Omega^{k}) \simeq 
\begin{cases}
0 & \: \: 0<k< d/2 \\ 
M_{\wedge^{k},k}({\G}) & \: \: d/2 \leq k < d \\ 
0 & \: \: k>d, \: k-d\not\in 2{\Z} \\ 
S\!_{d+m}({\G}, \det)\otimes {\C}\mathcal{S}_{n,m} & \: \: k=d+2m, \: 0\leq m \leq n 
\end{cases}
\end{equation*}
\end{theorem}

\begin{proof}
When $k\leq d$, we have 
$H^{0}(\bar{X}_{n}, \Omega^{k}) \simeq H^{0}(\bar{B}, \Omega^{k})$ by Proposition \ref{prop: extension}. 
Then $\bar{B}$ is a smooth projective model of the modular variety ${\G}\backslash \mathcal{D}$. 
By a theorem of Pommerening \cite{Pom}, 
the space $H^{0}(\bar{B}, \Omega^{k})$ for $k<d$ is isomorphic to the space of ${\G}$-invariant holomorphic $k$-forms on $\mathcal{D}$, 
which in turn is identified with the space $M_{\wedge^{k},k}({\G})$ 
of vector-valued modular forms of weight $(\wedge^{k},k)$ for ${\G}$ (\cite{Ma3}). 
The vanishing of this space in $0<k<d/2$ is proved in \cite{Ma3} Theorem 1.2 in the case when $L$ has Witt index $2$, 
and in \cite{Ma3} Theorem 1.5 (1) in the case when $L$ has Witt index $\leq 1$. 

The vanishing in the case $k>d$ with $k\not\equiv d$ mod $2$ follows from Proposition \ref{prop: extension}. 
Finally, we consider the case $k=d+2m$, $0\leq m \leq n$. 
By Proposition \ref{prop: extension}, we have a natural $\frak{S}_{n}$-equivariant isomorphism 
\begin{equation*}
H^{0}(\bar{X}_{n}, \Omega^{d+2m}) \simeq 
\bigoplus_{|\sigma|=m} H^{0}(\bar{X}_{\sigma}, K_{\bar{X}_{\sigma}})  
\end{equation*}
where $\frak{S}_{n}$ permutes the subsets $\sigma$ of $\{ 1, \cdots, n \}$. 
Here note that the stabilizer of each $\sigma$ acts on $H^{0}(\bar{X}_{\sigma}, K_{\bar{X}_{\sigma}})$ trivially by \eqref{eqn: KXsigma}. 
Therefore, as an $\frak{S}_{n}$-representation, the right hand side can be written as 
\begin{equation*}
H^{0}(\bar{X}_{m}, K_{\bar{X}_{m}}) \otimes \left( \bigoplus_{|\sigma|=m}{\C}\sigma \right) 
\simeq H^{0}(\bar{X}_{m}, K_{\bar{X}_{m}}) \otimes {\C}\mathcal{S}_{n,m}. 
\end{equation*}
Finally, we have 
$H^{0}(\bar{X}_{m}, K_{\bar{X}_{m}})\simeq S\!_{d+m}({\G}, \det)$ 
by \cite{Ma2} Theorem 3.1. 
\end{proof}

\begin{remark}
The case $k\geq d$ of Theorem \ref{thm: main lattice-pol} holds also when $d=1, 2$. 
We put the assumption $d\geq 3$ for the requirement of the Koecher principle from \cite{Pom}. 
Therefore, in fact, only the case $(d, k)=(2, 1)$ with Witt index $2$ is not covered. 
\end{remark}


\end{document}